\documentclass[12pt,oneside]{amsart}
\usepackage{geometry} 
\usepackage{amssymb,latexsym,amsmath,extarrows,amsthm, tikz}
\usepackage{pgfplots}
\usepackage{hyperref}
\usepackage{graphicx}
\usepackage{mathabx}
\usepackage{enumerate}
\usepackage{ulem} 

\DeclareMathAlphabet{\pazocal}{OMS}{zplm}{m}{n} 
\usepackage{color}
\definecolor{blue}{rgb}{0,0,1}

\definecolor{red}{rgb}{1,0,.2}

\theoremstyle{plain}
\newtheorem{thm}{Theorem}[section]
\newtheorem{lem}[thm]{Lemma}
\newtheorem{prop}[thm]{Proposition}
\newtheorem{defn}[thm]{Definition}
\newtheorem{remark}[thm]{Remark}
\newtheorem{note}[thm]{Note}

\numberwithin{equation}{section}

\newcommand{\R}{\ensuremath{\mathbb{R}}}
\newcommand{\Z}{\ensuremath{\mathbb{Z}}}

\usepackage{hyperref}

\providecommand{\emp}{\varnothing}

\newcommand{\de}{\delta}

\newcommand{\eps}{\epsilon}

\usepackage{hyperref}

\author[Alex McDonald]{Alex McDonald }
\address{Alex McDonald, Department of Mathematics, The Ohio State University}
\email{mcdonald.996@osu.edu}

\author[Krystal Taylor]{Krystal Taylor}
\address{Krystal Taylor, Department of Mathematics, The Ohio State University}
\email{taylor.2952@osu.edu}
\thanks{Taylor is supported in part by the Simons Foundation Grant 523555.}

\title{\parbox{14cm}{\centering{ 
Finite Point configurations in Products of Thick Cantor sets and a Robust Nonlinear Newhouse Gap Lemma
}}}
\begin{document}

\maketitle
\begin{abstract}
In this paper we prove that the set of tuples of edge lengths in $K_1\times K_2$ corresponding to a finite tree has non-empty interior, where $K_1,K_2\subset \R$ are Cantor sets of thickness $\tau(K_1)\cdot \tau(K_2) >1$.  
Our method relies on establishing that the pinned distance set is robust to small perturbations of the pin.
In the process, we prove a nonlinear version of the classic Newhouse gap lemma, and show that if $K_1,K_2$ are as above and $\phi: \R^2\times \R^2 \rightarrow \R $ is a function satisfying some mild assumptions on its derivatives, then there exists an open set $S$ so that $\bigcap_{x \in S} \phi(x,K_1\times K_2)$ has non-empty interior.  
\end{abstract}  
\maketitle

\section{Introduction} 
It is a simple consequence of the Lebesgue density theorem that subsets of $\R^n$ of positive Lebesgue measure contain  
a translated and scaled copy of every finite point set for an interval worth of scalings \cite{Steinhaus20}.  
Under more general assumptions on $E$, 
a problem of great current interest is that of describing the set of configurations that exists within $E$.  This includes 
finding conditions on the structure or size of $E$ that guarantee the existence of various patterns within $E$,
see, for instance, \cite{Bourgain86, CLP,  FKW90, IosLiu, IosMag, Krause, FY21, Ziegler06}, 
as well as the more quantitative question of describing the size of the set of similar copies of a given configuration \cite{GGIP, GIP, GIT21, GIT19, M21}.   
We focus on the latter question.  
  A particularly simple setting involves two point configurations and distances, which we describe using the following notation.

\begin{defn}[Distance sets]
Given a set $E\subset \R^d$, define the distance set of $E$ to be the set
\[
\Delta(E)=\{|x-y|:x,y\in E\}.
\]
For $x\in \R^d$, define the \textbf{pinned distance set} of $E$ at $x$ to be
\[
\Delta_x(E)=\{|x-y|:y\in E\}.
\]
\end{defn}

A classic problem in harmonic analysis and geometric measure theory is the Falconer distance problem, which asks how large the Hausdorff dimension of a set $E\subset \R^d$ must be to ensure that $\Delta(E)$ has positive Lebesgue measure \cite{Fal85}, or more generally non-empty interior. 
Falconer proved that $\dim_{\rm{H}}(E) >\frac{d}{2}$ is necessary and $\dim_{\rm{H}}(E) >\frac{d+1}{2}$ suffices for positive measure.   
The best known result in the plane improves this threshold to $\dim_{\rm{H}}(E) >\frac{5}{4}$ \cite{GIOW}.  In the non-empty interior direction, a result of Mattila and Sj\"olin shows that $\Delta(E)$ has non-empty interior provided that $\dim_{\rm{H}}(E)>\frac{d+1}{2}$; for a simple Fourier analytic proof of this fact, see also \cite{IMT12}. The first known results on the interior of pinned distance sets appeared in \cite[Corollary 2.12 \& Theorem 2.15]{STinterior}.  
\\

To pose questions about more complex patterns, the language of graph theory is useful.  We pause to state some basic definitions.

\begin{defn}[Graphs]
A (finite) \textbf{graph} is a pair $G=(V,E)$, where $V$ is a (finite) set and $E$ is a set of $2$-element subsets of $V$.  If $\{i,j\}\in E$ we say $i$ and $j$ are \textbf{adjacent} and write $i\sim j$.
\end{defn} 

Throughout this paper, we will always consider our vertex set to be $\{1,\dots,k+1\}$.  We are particularly interested in the following types of graphs.

 \begin{defn}[Chain and tree graphs]
The \textbf{$k$-chain} is the graph on vertex set $\{1,\dots,k+1\}$ with $i\sim j$ if and only if $|i-j|=1$.  A \textbf{tree} is a connected, acyclic graph; equivalently, a tree is a graph in which any two vertices are connected by exactly one path.  If $T$ is a tree, the \textbf{leaves} of $T$ are the vertices which are adjacent to exactly one other vertex of $T$.
  \end{defn} 
  
Note that all chains are trees.  We record the basic structural properties of trees as a proposition.

\begin{prop}[Tree structure]
\label{treestructure}
If $T$ is a tree with $k+1$ vertices, then $T$ has $k$ edges.  Moreover,  given such a tree $T$ there is a sequence of trees $T_1,...,T_k,T_{k+1}$ such that $T_1=T$, $T_{k+1}$ consists of only one vertex, and each $T_{i+1}$ is obtained from $T_i$ by removing one leaf and its corresponding edge.
\end{prop}

\begin{defn}[$G$ distance sets]\label{Gdist_defn}
Let $G$ be a graph on the vertex set $\{1,\dots,k+1\}$ with $m$ edges, and let $\sim$ denote the adjacency relation on $G$.  
Define the \textbf{$G$ distance set} of $E$ to be
\[
\Delta_G(E)=\{(|x^i-x^j|)_{i\sim j}:x^1,...,x^{k+1}\in E, x^i\neq x^j\},
\]
where $(a_{i,j})_{i\sim j}$ denotes a vector in $\R^m$ with coordinates indexed by the edges of $G$.
 \end{defn}

If $G$ is the $1$-chain, the set $\Delta_G(E)$ essentially coincides with the set $\Delta(E)$ (our definition of $G$ distance sets excludes degenerate configurations, so really $\Delta(E)=\Delta_G(E)\cup \{0\}$ for this choice of $G$).  
\\

With these definitions in place, the question becomes: what structural conditions on $E$ are needed to ensure $\Delta_G(E)$ has positive measure/non-empty interior?  
When $G$ is a $k$-chain, the problem was studied by Bennett, Iosevich, and the second author of this paper in \cite{BIT16}.  They prove that if the Hausdorff dimension of $E$ is greater than $\frac{d+1}{2}$, then $\Delta_G(E)$ has non-empty interior.  
Moreover, there is an interval $I\subset\R$ such that $I^k \subset \Delta_G(E)$, where $I^k$ denotes the $k$-fold Cartesian product of $I$.  
This result was later generalized by Iosevich and the second author to the case when $G$ is a tree \cite{ITtrees}.  
For additional progress on relating the Hausdorff dimension of a set to the interior of the set of configurations that it contains, see \cite{GIT21, GIT19}.
\\

In \cite{OuT20}, Ou and the second listed author show that when $T$ is a tree and $\dim_{\rm H}(E)>5/4$ for $E\subset \R^2$, then $\Delta_T(E)$ has positive Lebesgue measure.  However, when $\dim_H(E)<3/2$ it is not known whether $\Delta_T(E)$ has non-empty interior. 
We make progress on this question by replacing the Hausdorff dimension with an alternate notion of structure, namely that of Newhouse thickness.  Our precise definitions are as follows.

\begin{defn}[Cantor sets]
A \textbf{Cantor set} is a subset of $\R^d$ which is compact, perfect, and totally disconnected.
\end{defn}
When $K\subset \R$ is a Cantor set, we have the following notion of structure (also see see \cite{PTbook}). 
\begin{defn}[Thickness]
A \textbf{gap} of a Cantor set $K\subset \R$ is a connected component of the compliment $\R\setminus K$.  If $u$ is the right endpoint of a bounded gap $G$, for $b\in \R\cup \{\infty\}$, let $(a,b)$ be the closest gap to $G$ with the property that $u<a$ and $|G|\le b-a$.  The interval $(u,a)$ is called the \textbf{bridge} at $u$ and is denoted $B(u)$.  Analogous definitions are made when $u$ is a left endpoint.  The \textbf{thickness} of $K$ at $u$ is the quantity
\[
\tau(K,u):=\frac{|B(u)|}{|G|}.
\]
Finally, the thickness of the Cantor set $K$ is the quantity
\[
\tau(K):=\inf_u \tau(K,u),
\]
the infimum being taken over all gap endpoints $u$.
\end{defn}

Our goal is to prove Falconer type results for sets in the plane of the form $K\times K$, where $K$ is a Cantor satisfying $\tau(K)>1$.  
Before stating our results, we comment on the relationship between thickness and Hausdorff dimension.
One can easily construct a Cantor set $K$ with arbitrarily small thickness and Hausdorff dimension arbitrarily close to 1.  
This is due to the fact that thickness is defined using an infimum, so one can construct a thin Cantor set by simply ensuring one bridge is much smaller than the corresponding gap.  
More precisely, for any $\de>0$ and $1< N< \de^{-1}$ we can construct $K$ as a subset of $[0,\de]\cup[N\de,1]$.  
It is clear that Cantor sets of this form can attain any Hausdorff dimension in $[0,1]$. 
Considering the gap $(\de, N\de)$ and corresponding bridge $[0,\de]$, we conclude that $\tau(K)\leq \frac{1}{N-1}$.   \\

On the other hand, large thickness implies large Hausdorff dimension.  Specifically, one can prove the bound (see \cite[pg. 77]{PTbook}):
$$\dim_{\rm H}(K) \geq \frac{\log{2}}{\log{\left( 2 + \frac{1}{\tau(K)} \right) }}.$$
In particular, when $K$ is a Cantor set of thickness $\tau(K)>1$ then we have
\[
\dim_{\rm H}(K\times K)> \frac{2\cdot\log 2}{\log 3}\approx 1.26
\]

\subsection{\textit{Main results}}
 Recall that if $E\subset \R^2$ and $\dim_{\rm H}(E)>\frac{5}{4}$, then $\Delta_T(E)$ has positive Lebesgue measure. 
 However, when $\dim_H(E)<\frac{3}{2}$ it is not known whether $\Delta_T(E)$ has non-empty interior.
 Our first result is as follows.

\begin{thm}[Interior of tree distance sets]
\label{distancetrees}
Let $K_1,K_2$ be Cantor sets satisfying $\tau(K_1)\cdot \tau(K_2) > 1$
 For any finite tree $T$, the set $\Delta_T(K_1\times K_2)$ has non-empty interior.
\end{thm}

 We note that Theorem \ref{distancetrees} is an extension of the work in \cite{STinterior} of Simon and the second listed author, where it is shown $\Delta_x(K\times K)$ has non-empty interior provided that $K$ is a Cantor set satisfying $\tau(K)> 1$.  
 Theorem \ref{distancetrees} also holds if the Euclidean norm is replaced with more general norms; see Theorem \ref{phi_distancetrees} below.\\

Beyond trees, 
the existence of patterns in thick subsets of $\R^d$ with rigid structure was investigated in \cite{Yav20} when $d=1$, and in \cite{FY21} when $d\geq 1$.  
In \cite{Yav20}, it is shown that, given any compact set $C$ in $\mathbb{R}$ with thickness
$\tau$, there is an explicit number $N(\tau)$ such that $C$ contains a
translate of all sufficiently small 
similar copies of every finite set in $\mathbb{R}$ with at most $N(\tau)$
elements.
Higher dimensional analogues of these results are subsequently given in \cite{FY21}. 
The only drawback is that the theorems in \cite{Yav20, FY21} assume very large thickness; moreover, the threshold depends on the size of the configuration one wants to find. 
For instance, in order to ensure $N(\tau)\geq 3$, one needs $\tau$ at least on the order of $10^9$. 
 In contrast, our results for trees apply to any Cantor sets of thickness greater than $1$, regardless of how large the tree is.  
\\

Another Falconer type problem which has received much attention is obtained by replacing the Euclidean distance with other geometric quantities, notably dot products.  
We make the following definition.

\begin{defn}[Dot product sets]
Given $E\subset\R^d$, the \textbf{dot product set} of $E$ is the set
\[
\Pi(E)=\{x\cdot y: x,y\in E\}.
\]
We also consider the \textbf{pinned dot product set}
\[
\Pi_x(E)=\{x\cdot y:y\in E\}.
\]
Finally, given a graph $G$ on vertices $\{1,...,k+1\}$, define
$$\Pi_{G}(E)=\{(x^i\cdot x^j)_{i\sim j}:x^1,...,x^{k+1}\in E\}.$$
\end{defn}

When $E\subset \R^d$ is a set of sufficient Hausdorff dimension, the dot product set is treated in \cite[Theorem 1.8]{EIT11}.
In particular, it is shown there that if $\dim_{\rm H}(E) >\frac{d+1}{2}$, then $\Pi(E)$ has positive measure.  The related set $\{x^\perp\cdot y:x,y\in E\}$, where $x^\perp=(-x_2,x_1)$ when $d=2$, is the set of (signed) areas of parallelograms spanned by points of $E$.  Similar to the above definition, for any graph $G$ one can consider the vector which encodes all areas determined by points $x^i,x^j$ such that $i\sim j$.  This problem was investigated by the first author in \cite{M21} in the case where $G$ is a complete graph, and the analogous problem in higher dimensions was studied by the first author and Galo in \cite{GM21}. \\
  
In the setting where one is considering sets $E$ with large Hausdorff dimension, the proofs of distance and dot product results are generally similar in complexity.  However, in the setting where $E= K\times K$, where $K$ is a sufficiently thick Cantor set, the dot product problem is considerably more straightforward than to the distance problem.  We nevertheless record the result here and provide its proof in Section \ref{KeyLemmas} as a demonstration of how our techniques vary in these two regimes.

\begin{thm}[Interior of tree dot product sets]
\label{dotproducttrees}
Let $K$ be a Cantor set satisfying $\tau(K)\geq 1$.  For any finite tree $T$, the set $\Pi_{T}(K\times K)$ has non-empty interior.
\end{thm}

Our next main theorem concerns the standard middle thirds Cantor set, which we will denote $C_{1/3}$ throughout this paper.  Note that Theorem \ref{distancetrees} does not apply to $C_{1/3}$, as the hypothesis of that theorem is $\tau(K)>1$ and clearly $\tau(C_{1/3})=1$.  While we do not expect that Theorem \ref{distancetrees} can be extended to the $\tau(K)=1$ case in general, this weaker thickness condition together with the self similarity of $C_{1/3}$ allow us to modify the proof in that case.  The result is as follows.
\begin{thm}[Interior of $T$ distance sets in the middle third Cantor set]
\label{trees_thrm_middle_third}
For any finite tree $T$, the set $\Delta_{T}(C_{1/3}\times C_{1/3})$ has non-empty interior.  
\end{thm}
\vskip.125in

\subsection{\textit{General distance trees}}
Having established results for the Euclidean distance and dot products, we turn to the more general setting of $(G,\phi)$ distance trees.

\begin{defn}[$ (G,\phi)$ distance sets]\label{Gdist_defn}
Let $G$ be a graph on the vertex set $\{1,\dots,k+1\}$ with $m$ edges, and let $\sim$ denote the adjacency relation on $G$.  
Given a function $\phi:\R^d\times \R^d\rightarrow \R$, define the \textbf{$(G,\phi)$-distance set} of $E$ to be 
\[ 
\Delta_{(G, \phi)}(E)=\{ (\phi( x^i,x^j))_{i\sim j}:x^1,...,x^{k+1}\in E, x^i\neq x^j\}.
\]
 \end{defn}

We require the following derivative condition on $\phi$.
\begin{defn}[Derivative condition]
\label{deriv_cond}
Let $\phi: \R^2 \times \R^2\rightarrow \R$ be a $C^1$ function 
on $A\times B$, for open sets $A, B\subset \R^2$.
We say that $\phi$ satisfies the \textbf{derivative condition} on $A\times B$ if for each $x\in A$, if $\varphi_x(y)= \phi(x,y)$,  then 
the partial derivatives of $\varphi$ are bounded away from zero on $B$. 
\end{defn}

\begin{note}
Note that the derivative condition is satisfied, for instance, by $\phi(x,y) = |x-y|_p$, the $p$-norm, whenever $p\geq1$, and $\phi(x,y) = x\cdot y $ for appropriate choices of $A$ and $B$.  
\end{note}

\begin{thm}[Interior of $ (T,\phi)$ distance sets]
\label{phi_distancetrees}
Let $K_1,K_2$ be Cantor sets satisfying $\tau(K_1)\cdot \tau(K_2) > 1$.
Suppose $\phi: \R^2 \times \R^2\rightarrow \R$ 
satisfies the derivative condition on $A\times B$, for open $A,B\subset \R^2$, each of which intersects $K_1\times K_2$.  Then,  
for any finite tree $T$, the set $\Delta_T(K_1\times K_2)$ has non-empty interior.
\end{thm}

The proof of Theorem \ref{phi_distancetrees} is given in Section \ref{general_proof} and relies on the mechanism in Section \ref{mechanismproof}. 
\subsection{\textit{Method of proof}}
We now discuss the strategy for proving our results.  The starting point is the following classical result known as the Newhouse Gap Lemma \cite[pg. 61]{PTbook}.
\begin{lem}[Newhouse Gap Lemma]
Let $K_1,K_2\subset \R$ be Cantor sets satisfying $\tau(K_1)\tau(K_2)\geq 1$.  Suppose further that neither of the sets $K_1,K_2$ is contained in a single gap of the other.  Then, $K_1\cap K_2\neq\emp$.
\end{lem}

In practice, if $K_2$ is contained in the convex hull of $K_1$ it can be difficult to check whether $K_2$ is contained in a gap of $K_1$.  We will often use a special case of this condition which is easier to check.  To do this we first introduce some terminology.
\begin{defn}[Linked sets]
Two open, bounded intervals $I,J\subset\R$ are said to be \textbf{linked} if they have non-empty intersection, but neither is contained in the other.  Bounded (not necessarily open) intervals are linked if their interiors are linked.  Finally, two bounded sets $K_1,K_2\subset\R$ are linked if their convex hulls are linked.
\end{defn}
\begin{prop}[Special case of Newhouse Gap Lemma]
\label{specialcase}
Let $K_1,K_2\subset\R$ be linked Cantor sets satisfying $\tau(K_1)\cdot \tau(K_2)\geq 1$.  Then, $K_1\cap K_2\neq\emp$.
\end{prop}
Now, given a fixed point $x\in \R^2$ and distance $t\in \R$, we have $|x-y|=t$ if $y_2=g_{x,t}^{\text{dist}}(y_1)$, where
\[
g_{x,t}^{\text{dist}}(z)=x_2+\sqrt{t^2-(z-x_1)^2}.
\]
Likewise, we have $x\cdot y=t$ if $y_2=g_{x,t}^{\text{dot}}(y_1)$, where
\[
g_{x,t}^{\text{dot}}(z)=\frac{t}{x_2}-\frac{x_1}{x_2}z.
\]
We are therefore interested in applying the Newhouse Gap Lemma to find a point in the intersection $K\cap g(K)$ for an appropriate function $g$.  In the case of the dot product this is simple, since affine transformations preserve thickness and we therefore only have to find values of $x,t$ for which $K$ and $g_{x,t}^{\text{dot}}(K)$ are linked.  Matters are more difficult for distances since, in general, smooth functions do not necessarily preserve thickness.  In \cite{STinterior} it is proved that if $g$ is continuously differentiable and $I$ is a sufficiently small interval on which $g'$ is bounded away from zero, then the thickness of $g(K\cap I)$ is not too much smaller than that of $K$.  This allows one to prove that the pinned sets $\Delta_x(K\times K)$ and $\Pi_x(K\times K)$ have non-empty interior if $\tau(K)>1$. \\

Using this strategy, we prove that there exists an interval $I$ which is contained in $\Delta_x(K\times K)$, and such that $I$ remains in $\Delta_x(K\times K)$ if the pin $x$ is perturbed a small amount.  Indeed, let $g_{x,t}$ denote either of the functions defined above.  For fixed $z$, the quantity $g_{x,t}(z)$ is continuous in the parameters $x$ and $t$.  Therefore, the condition that $K$ and $g_{x,t}(K)$ are linked is an open condition.  In Section \ref{KeyLemmas} we prove these ``pin wiggling" lemmas for each of the functions needed in our theorems.  Once we have proved such a lemma, we can convert it into a theorem about trees using a mechanism which is discussed in Section \ref{mechanismproof}. 
 

\section{Converting Pin Wiggling Lemmas into Tree Theorems}
\label{mechanismproof}
In Section \ref{KeyLemmas}, we will prove lemmas showing that not only do pinned distance and dot product sets contain intervals, but that there is a single interval which works for all such sets obtained by wiggling the pin a small amount.  The goal of this section is to state and prove a theorem which gives us a mechanism to convert pin wiggling lemmas to our main theorems.  The setup is as follows.  Let $\phi:\R^2\times \R^2\to \R$ be any function; for example, to prove Theorem \ref{distancetrees} we use the function $\phi(x,y)=|x-y|$.  Given a point $x$ and a set $E$, we use the notation
\[
\phi(x,E):=\{\phi(x,y):y\in E\}.
\]
A \textbf{pin wiggling lemma} is a lemma which says, under some assumptions on $E$, that there is a single interval $I$ contained in $\phi(x,E)$ for a range of $x$; equivalently, the set
\[
\bigcap_{x\in S}\phi(x,E)
\]
has non-empty interior for some neighborhood $S$ of pins.  Now, given such a function $\phi$ and a tree $T$ on vertices $\{1,\dots,k+1\}$, define
\[
\Phi(x^1,\dots,x^{k+1})=(\phi(x^i,x^j))_{i\sim j}.
\]
Thus, the sets $\Delta_T(E)$ and $\Pi_T(E)$ are the images of $E^{k+1}$ under $\Phi$ for $\phi(x,y)=|x-y|$ and $\phi(x,y)=x\cdot y$, respectively.  Our main theorems are therefore giving conditions under which the sets $\Phi(E^{k+1})$ have non-empty interior.  With this setup, the conversion mechanism is as follows.
\begin{thm}[Tree building mechanism]
\label{mechanism}
Fix a map $\phi:\R^2\times \R^2\to\R$ and a tree $T$ on vertices $\{1,\dots,k+1\}$, and consider the map $\Phi:(\R^2)^{k+1}\to\R^k$ defined by
\[
\Phi(x^1,\dots,x^{k+1})=(\phi(x^i,x^j))_{i\sim j},
\]
where $\sim$ denotes the adjacency relation of the graph $T$.  
Let $K_1,K_2$ be Cantor sets satisfying $\tau(K_1)\cdot \tau(K_2) > 1$,
and let $x^1,\dots,x^{k+1}\in K_1\times K_2$ be distinct points.  Suppose that for any Cantor sets $\widetilde{K_j} \subset K_j$,
there exist open neighborhoods $S_i$ of $x^i$ such that the set
\[
\bigcap_{x\in S_i}\phi(x,\widetilde{K_1}\times \widetilde{K_2})
\]
has non-empty interior.  Then, $\Phi((K_1\times K_2)^{k+1})$ has non-empty interior.  Moreover, $\Phi(x^1,\dots,x^{k+1})$ is in the closure of $\Phi((K_1\times K_2)^{k+1})^\circ$. 

\end{thm}

\begin{proof}
For the purpose of ensuring non-degeneracy, 
let $2\eps>0$ denote the minimal distance:
$$\eps= \frac{1}{2} \min \left\{ |x^i- x^{j}|: i \neq j \in \{ 1, 2, \dots, k+1\} \right\}>0,$$
and, for each $i=1,2, \dots, k+1$, define the $\eps$-box about $x^i$ by 
\begin{align*}
C(x^i, \eps)
=& x^i + [- \eps,\eps]^2\\
=& [x^i_1 - \eps, x^i_1+\eps] \times [x^i_2- \eps, x^i_2+\eps]\\
=& C_1(x^i, \eps) \times C_2(x^i, \eps),
\end{align*}
where $C_1(x^i, \eps), C_2(x^i, \eps)$ are the closed $\eps$-intervals about the coordinates of $x^i$ (Figure \ref{boxes}).
\begin{figure}[ht]
\centering
\begin{tikzpicture}
\begin{axis}[xmin=-.1,xmax=5,ymin=-.1,ymax=5,axis x line=center, axis y line=center, xtick={0}, xticklabels={}, ytick={0},yticklabels={}]

\addplot[mark=*] coordinates {(1,2)};
\addplot[mark=*] coordinates {(3,2.5)};
\addplot[mark=*] coordinates {(4,.5)};
\addplot[mark=*] coordinates {(4.5,3)};
\addplot[mark=*] coordinates {(1.5,4)};

\addplot[dashed] coordinates {(1,2)(1.5,4)};
\addplot[dashed] coordinates {(1,2)(3,2.5)};
\addplot[dashed] coordinates {(3,2.5)(4.5,3)};
\addplot[dashed] coordinates {(3,2.5)(4,.5)};

\addplot[thick] coordinates {(.8,1.8)(.8,2.2)(1.2,2.2)(1.2,1.8)(.8,1.8)};
\addplot[thick] coordinates {(2.8,2.3)(2.8,2.7)(3.2,2.7)(3.2,2.3)(2.8,2.3)};
\addplot[thick] coordinates {(3.8,.3)(3.8,.7)(4.2,.7)(4.2,.3)(3.8,.3)};
\addplot[thick] coordinates {(4.3,2.8)(4.3,3.2)(4.7,3.2)(4.7,2.8)(4.3,2.8)};
\addplot[thick] coordinates {(1.3,3.8)(1.3,4.2)(1.7,4.2)(1.7,3.8)(1.3,3.8)};

\end{axis}

\end{tikzpicture}
\caption{Boxes $C(x^i,\eps)$ around points $x^1,...,x^5$}
\label{boxes}
\end{figure}
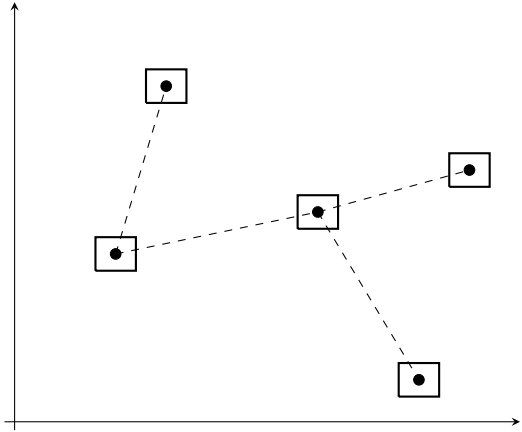

Next, choose any leaf of $T$; without loss of generality we may assume we have labeled the vertices so that $k+1$ is our leaf.  Let $i$ denote the unique vertex which satisfies $i\sim k+1$.  Let $\widetilde{K_j} = K_j \cap C_j(x^{k+1}, \eps)$.   By assumption, there exists a neighborhood $S_{i}$ of $x^{i}$ so that the set
\begin{equation}\label{first_iter_eq}
 \bigcap_{x\in S_{i}}\phi(x,\widetilde{  K_1 }\times \widetilde{K_2})
\end{equation}
has non-empty interior.  Further, we may assume $S_{i}\subset C(x^{i},\eps)$, which guarantees that the points in $S_{i}$ and points in $\widetilde{K_1}\times \widetilde{K_2}\subset C(x^{k+1}, \eps)$ are distinct. 
Moreover, we can choose $\eps_{2} \in (0, \eps]$ so that $C(x^i,\eps_{2}) \subset S_{i}$, and hence \eqref{first_iter_eq} still holds with $C(x^{i},\eps_2) $ in place of $S_{i}$.  For simplicity, we replace each of the  $\eps$-boxes about $x^{1}, \dots, x^{k+1}$ by potentially smaller boxes $C(x^{j}, \eps_2)$ for each $j \in \{1, \dots, k+1\}$.
\\

To conclude, let $E_{i} = C(x^{i},\eps_2)\cap (K_1\times K_2)$, let $T_2$ be the tree obtained from $T$ by removing the vertex $k+1$ and its corresponding edge, and let $\Phi_2$ be the function as in the statement of the theorem, corresponding to the tree $T_2$.  We have proved there exists a non-empty open interval $I_{1}$ so that 
\[
\Phi(E_1\times\cdots\times E_{k+1})\supset \Phi_2(E_1\times\cdots\times E_k)\times I_1.
\]
Running this argument successively on each of the trees $T_1,T_2,...,T_k$ as in Proposition \ref{treestructure}, we conclude that $\Phi((K_1\times K_2)^{k+1})$ contains a set of the form $I_1\times \dots\times I_k$ for non-empty open intervals $I_1, \dots, I_k$.  By construction, it is clear that $\Phi(x^1,...,x^{k+1})$ is in the closure of $I_1\times\cdots\times I_k$.
\end{proof}

From the statement of Theorem \ref{mechanism}, we see that we may start with any points $x^1,\dots,x^{k+1}\in K_1\times K_2$ and obtain an open box near $\Phi(x^1,\dots,x^{k+1})$, provided we can prove pin wiggling lemmas around those points.  We will refer to the starting points $x^1,\dots,x^{k+1}$ as a \textbf{skeleton}.

\begin{remark}
It would be interesting to find an interval $I$ such that $I^k\subset \Delta_T(K\times K)$, as is done in \cite{BIT16, ITtrees} in the large Hausdorff dimension context.  By Theorem \ref{mechanism} below, this amounts to showing there exists a skeleton $x^1,...,x^{k+1}\in K\times K$ such that to two points share a coordinate, and the distances $|x^i-x^j|$ are constant for $i\sim j$.  It is not clear how to do this in general.  In the special case where $T$ is a $k$-chain, it is sufficient (but not necessary) that $K$ contains a length $k+1$ arithmetic progression.  Given an arithmetic progression $a_1,...,a_{k+1}\in K$, we could then take $x^i=(a_i,a_i)$.  A result of Yavicoli \cite{Yav20} shows that long arithmetic progressions exist in (very) thick Cantor sets.  However, the required lower bound on thickness is much larger then the $\tau(K)>1$ assumption in our results; to ensure even a $3$-term arithmetic progression, one needs $\tau(K)$ at least on the order of $10^9$.  Moreover, let $C_\eps$ denote the middle-$\eps$ Cantor set, obtained by starting with the unit interval and at each stage deleting the middle $\eps$ proportion from the remaining intervals.  Broderick, Fishman, and Simmons \cite{BFS19} prove that there is no arithmetic progression in $C_\eps$ of length greater than $\frac{1}{\eps}+1$ for $\eps$ sufficiently small.  Therefore, for any $k\in \Z$ sufficiently large, the set $C_{2/k}$ is a Cantor set with thickness $\frac{k-2}{4}$ and no $(k+1)$-term arithmetic progression.  
This means there is no hope for a thickness threshold which is uniform in $k$.  One can still hope to find a common interval by using the $2$-dimensionality of $K\times K$ instead of hoping to take a sequence of points along the diagonal, but it is not clear how to find the necessary skeleton.
\end{remark}


\section{Pin Wiggling Lemmas in Various Contexts}
\label{KeyLemmas}
The proofs in this section are presented in increasing order of complexity. 
\subsection{\textit{Proof of Theorem \ref{dotproducttrees}}}
We begin by proving our result on dot product trees.  As discussed in the introduction, dot products are much simpler than distances because thickness is preserved under affine transformations.  As a consequence, Theorem \ref{dotproducttrees} is the simplest of our results.\\

Theorem \ref{dotproducttrees} is an immediate consequence of Theorem \ref{mechanism} and the following lemma.
\begin{lem}[Pin Wiggling for Dot Products]
\label{dotproductkeylemma}
Let $K_1,K_2$ be Cantor sets satisfying $\tau(K_1)\cdot \tau(K_2) \geq 1$.  Let $\ell_j$ denote the length of the convex hull of $K_j$.
\begin{enumerate}[(i)]
\item For any $x=(x_1,x_2)\in \R^2$ with both coordinates nonzero, the set $\Pi_x(K_1\times K_2)$ contains an interval of length at least $\ell\cdot\min(|x_1|,|x_2|)$
\item Let $x^0=(x_1^0,x_2^0)\in \R^2$ be a point with both coordinates nonzero.  Let $Q$ be the square centered at $x^0$ with side length $2\delta$, and assume $\delta<\frac{1}{3}\min(|x_1^0|,|x_2^0|)$.  The set
\[
\bigcap_{x\in Q}\Pi_x(K_1\times K_2)
\] 
contains an interval of length at least $\ell\cdot(\min(|x_1^0|,|x_2^0|)-3\delta)$.
\end{enumerate}
\end{lem}

\begin{proof}
For any $x=(x^1, x^2)\in \R^2$, we have $t\in \Pi_x(K_1\times K_2)$ if and only if $(t-x_1K_1)\cap (x_2K_2)\neq\emp$.  Since $\tau(t-x_1K_1)\cdot \tau(x_2K_2)=\tau(K_1)\cdot \tau(K_2)\geq 1$, by the Newhouse Gap Lemma, this intersection will be non-empty for any $x$ and $t$ such that the sets $(t-x_1K_1)$ and $x_2K_2$ are linked.  Denote the convex hull of $x_jK_j$ by $[a_j,a_j+\ell|x_j|]$, and without loss of generality assume $|x_1|\geq |x_2|$.  The sets $t-x_1K$ and $x_2K$ are linked whenever 
\[
a_1+a_2+\ell |x_1|<t<a_1+a_2+\ell |x_1|+\ell |x_2|.
\tag{$*$}
\]
The set of $t$ satisfying ($*$) is an interval of length $\ell |x_2|$, and so (i) follows immediately.  To prove (ii), assume $(x_1,x_2)\in Q$ and therefore $|x_j-x_j^0|<\delta$ for each $j$.  The value $t$ satisfies ($*$) for all such $x_1,x_2$ provided
\[
a_1+a_2+\ell |x_1^0|+\ell\delta<t<a_1+a_2+\ell |x_1^0|+\ell |x_2^0|-2\ell\delta.
\]
This inequality determines an interval of length $\ell(|x_2^0|-3\delta)$.

\end{proof}

Note that when we apply Theorem \ref{mechanism}, we can start with any skeleton $x^1,\dots,x^{k+1}\in K\times K$ such that none of the points $x^i$ are on the axes.

\subsection{\textit{Proof of Theorem \ref{distancetrees}}}
As in the previous section, the proof will rely on the mechanism established in Theorem \ref{mechanism} coupled with a pin wiggling lemma.  The difference is that the lemma of this section will not follow directly from the linear theory and some preliminary set up is required.  

First, observe that given a pin $x\in \R^2$ and distance $t\in \R$, we have
$t\in \Delta_x(K\times K)$ whenever $y_2=g_{x,t}(y_1)$ for some $y = (y_1, y_2) \in K\times K$, where
\[
g_{x,t}(z)=x_2+\sqrt{t^2-(z-x_1)^2}.
\]

We would like to apply the Newhouse Gap Lemma to $K$ and $g_{x,t}(K)$ to prove a pin wiggling lemma for distance sets (Lemma \ref{distanceskeylemma} below), then conclude that Theorem \ref{distancetrees} follows (by Theorem \ref{mechanism}).  However, there is no thickness assumption on $K$ which would guarantee $\tau(g_{x,t}(K))\geq 1$, so we cannot apply Newhouse directly.  However, if $I$ is a sufficiently small interval about a non-singular point of $g_{x,t}$, then the thickness of $g_{x,t}(K\cap I)$ is not too much smaller than that of $K$.  This can be proved using a generalization of thickness which was introduced in \cite{STinterior}, which we describe here.

\begin{defn}[$\epsilon$-thickness]
Let $K\subset\R$ be a Cantor set, let $u$ be a right endpoint of a bounded gap, and let $\eps>0$.  Let $(a,b)$ be the closest gap to $G$ with the property that $a>u$ and $(b-a)>(1-\eps)|G|$ The \textbf{$\eps$-bridge} of $u$, denoted $B_\eps(u)$, is the interval $(u,a)$.  We make analogous definitions for left endpoints.  The \textbf{$\eps$-thickness} of $K$ at $u$ is the quantity
\[
\tau_\eps(K,u):= \frac{|B_\eps(u)|}{|G|}.
\]
Finally, the $\eps$-thickness of the Cantor set $K$ is the quantity
\[
\tau_\eps(K):=\inf_u \tau_\eps(K,u),
\]
the infimum being taken over all gap endpoints $u$.
\end{defn}
We record some easily verifiable properties of $\eps$-thickness in the following proposition.
\begin{prop}[$\eps$-thickness converges to regular thickness]
\label{basics}
Let $K\subset \R$ be a Cantor set.
\begin{enumerate}[(i)]
\item If $\eps_1<\eps_2$ then $\tau_{\eps_1}(K)\geq \tau_{\eps_2}(K)$.
\item $\tau_\eps(K)\to \tau(K)$ as $\eps\to 0$.
\end{enumerate}
\end{prop}
With these definitions in place, we can prove that the image of a thick Cantor set must at least contain a thick Cantor set.  More precisely, we have the following lemma, which is essentially Lemma 3.8 in \cite{STinterior}.  We include a proof here for completeness.
\begin{lem}[Thickness of the image is nearly preserved]
\label{imagethickness}
Let $K\subset\R$ be a Cantor set, let $u$ be a right endpoint of some gap of $K$, and let $g$ be a function which is continuously differentiable on a neighborhood of $u$ and satisfies $g'(u)\neq 0$.  For every $\eps>0$, there exists $\delta>0$ such that
\[
\tau(g(K\cap [u,u+\delta]))>\tau_\eps(K)(1-\eps).
\]

\end{lem}
\begin{proof}
Fix $\eps>0$.  By continuity of $g'$, we may choose $\delta$ such that for all $x_1,x_2\in[u,u+\delta]$ we have 
\[
\left|\frac{g'(x_1)}{g'(x_2)}-1\right|<\eps.
\]
Note that our choice of $\delta$ guarantees that $g$ is monotone on the interval $[u,u+\delta]$, so for any subinterval $I$ the mean value theorem guarantees the existence of some $x_I\in I$ such that $|g(I)|=|I|\cdot |g'(x_I)|$.  Let $v$ be the endpoint of some gap $G$ in $K\cap [u,u+\delta]$.  We first observe $g(B_\eps(v))\subset B_{\eps^2}(g(v))$.  To prove this, note that any gap in $g(K\cap [u,u+\delta])$ is the image of a gap in $K\cap [u,u+\delta]$.  Therefore, it suffices to prove that any gap $H\subset B_\eps(v)$ satisfies $g(H)<(1-\eps^2)g(G)$.  We have
\begin{align*}
|g(H)|&=|H|\cdot |g'(x_H)| \\
&<(1-\eps)|G|\cdot\frac{|g'(x_H)|}{|g'(x_G)|}\cdot |g'(x_G)| \\
&<(1-\eps)|G|\cdot (1+\eps)\cdot |g'(x_G)| \\
&=(1-\eps^2)|g(G)|.
\end{align*}
It follows that the thickness of our image at the point $g(v)$ satisfies
\begin{align*}
\tau_{\eps^2}(g(K\cap[u,u+\delta]),g(v)) &\geq \frac{|g(B_\eps(v))|}{|g(G)|}\\
&=\frac{|B_\eps(v)|}{|G|}\cdot \left|\frac{g'(x_{B_\eps(v)})}{g'(x_G)}\right| \\
&\geq \tau_\eps(K\cap[u,u+\delta],v)\cdot(1-\eps).
\end{align*}
Taking the infimum over $v$, we have
\[
\tau(g(K\cap[u,u+\delta]))>\tau_{\eps^2}(g(K\cap[u,u+\delta]))\geq \tau_\eps(K)(1-\eps).
\]
This is the first claim in the statement of the theorem.  The second then follows from Proposition \ref{basics} and the assumption that $\tau(K)>1$.
\end{proof}

We are now prepared to state and prove the key lemma for distances. 

\begin{lem}[Pin Wiggling for Distances]
\label{distanceskeylemma}
Let $K_1,K_2$ be Cantor sets satisfying $\tau(K_1)\cdot \tau(K_2) > 1$.  For any $x^0\in \R^2$, there exists an open $S$ about $x^0$ so that
$$\bigcap_{x\in S} \Delta_{x}(K_1\times K_2)$$
has non-empty interior.  
\end{lem}
\begin{proof}
For $(x,t)\in\R^2\times (0,\infty)$, define
\begin{equation}\label{baby_g}
g_{x,t}(z)=x_2+\sqrt{t^2-(z-x_1)^2},
\end{equation}

and observe that $t\in\Delta_{x}(K_1\times K_2)$ provided $K_2\cap g_{x,t}(K_1)\neq\emp$.  Let $u_j$ be a right endpoint of a bounded gap of $K_j$, and without loss of generality assume $u_j>x_j^0$, where $x^0 = (x^0_1, x^0_2)$.  We choose small subsets $\widetilde{K_j}\subset K_j$ with left endpoint $u_j$, and focus on the box $\widetilde{K_1}\times \widetilde{K_2}$  (Figure \ref{box}).  Let $t_0=|x^0-(u_1,u_2)|$.  Let $\widetilde{K_j}=K_j\cap[u_j,u_j+\delta_j]$ for some small $\delta_1,\delta_2$.  In particular, we choose $\delta_j>0$ small enough that $\tau(\widetilde{K_2})\cdot \tau(g(\widetilde{K_1}))>1$ (this is possible by Lemma \ref{imagethickness}), and so that $\widetilde{K_1}$ is in the domain of $g_{x,t}$ whenever $(x,t)$ is sufficiently close to $(x^0,t_0)$.  We will also assume $u_j+\delta_j\in K_j$.
\begin{figure}[ht]
\centering
\begin{minipage}[b]{0.45\linewidth}
\begin{tikzpicture}
\begin{axis}[xmin=0,xmax=5,ymin=0,ymax=5,axis x line=center, axis y line=center, xtick={2,3,4}, xticklabels={$x_1^0$,$u_1$,$u_1+\delta_1$}, ytick={1,3,4},yticklabels={$x_2^0$,$u_2$,$u_2+\delta_2$}]

\addplot[mark=*] coordinates {(2,1)};

\addplot[dashed] coordinates {(3,0)(3,5)};
\addplot[dashed] coordinates {(4,0)(4,5)};
\addplot[dashed] coordinates {(0,3)(5,3)};
\addplot[dashed] coordinates {(0,4)(5,4)};

\addplot[thick] coordinates {(3,3)(4,3)(4,4)(3,4)(3,3)};

\end{axis}

\end{tikzpicture}
\caption{The box containing $\widetilde{K_1}\times \widetilde{K_2}$}
\label{box}
\end{minipage}
\qquad
\begin{minipage}[b]{0.45\linewidth}
\begin{tikzpicture}
\begin{axis}[xmin=0,xmax=5,ymin=0,ymax=5,axis x line=center, axis y line=center, xtick={2,3,4}, xticklabels={$x_1^0$,$u_1$,$u_1+\delta_1$}, ytick={1,3,4},yticklabels={$x_2^0$,$u_2$,$u_2+\delta_2$}]

\addplot[domain=0:5,samples=100,red] {1+sqrt(5-(x-2)^2)};
\addlegendentry{$g_{x^0,t_0}$}

\addplot[domain=0:5,samples=100,blue] {1+sqrt(5.56-(x-2)^2)};
\addlegendentry{$g_{x^0,t}$}

\addplot[dashed] coordinates {(3,0)(3,3.13)};
\addplot[dashed] coordinates {(4,0)(4,2.25)};

\addplot[dashed] coordinates {(0,2.25)(4,2.25)};
\addplot[dashed] coordinates {(0,3.13)(3,3.13)};

\addplot[mark=*] coordinates {(2,1)};

\addplot[] coordinates {(3,3)(4,3)(4,4)(3,4)(3,3)};

\end{axis}

\end{tikzpicture}
\caption{Graphs of $g$}
\label{graphs}
\end{minipage}
\end{figure}

By the Newhouse Gap Lemma (Proposition \ref{specialcase}), we will have $\widetilde{K_2}\cap g_{x,t}(\widetilde{K_1})\neq\emp$ whenever the parameters $(x,t)$ are such that $\widetilde{K_2}$ and $g_{x,t}(\widetilde{K_1})$ are linked.  To this end, consider the set
\[
U=\{(x,t)\in\R^2\times\R:g_{x,t}(u_1+\delta_1)<u_2<g_{x,t}(u_1)<u_2+\delta_2\}.
\]
By construction, for any $(x,t)\in U$ we have $K_2$ and $g_{x,t}(K_1)$ linked (Figure \ref{graphs}) and hence $t\in\Delta_x(K_1\times K_2)$.  We claim that $U$ is an open set containing a point of the form $(x^0,t)$ for some $t$.  Lemma \ref{distanceskeylemma} follows from this claim, as we can then take open neighborhoods $S,T$ of $x^0,t$ respectively such that 
\[
T\subset \bigcap_{x\in S} \Delta_{x}(K_1\times K_2).
\]
To prove the claim, we first observe that $U$ is open, since for fixed $z$ the quantity $g_{x,t}(z)$ is a continuous function of $(x,t)$.  To finish the proof, we must find a $t$ such that $(x^0,t)\in U$.  By construction, we have $g_{x^0,t_0}(u_1)=u_2$.  Since the quantity $g_{x,t}(z)$ is strictly increasing in $t$, for any $t>t_0$ we will have $g_{x^0,t}(u_1)>u_2$.  On the other hand, by continuity in $t$ we will also have $g_{x^0,t}(u_1+\delta_1)<u_2$ and $g_{x^0,t}(u_1)<u_2+\delta_2$ whenever $t$ is sufficiently close to $t_0$.  Therefore, we can choose $t$ with the property that $(x^0,t_0)\in U$.

\end{proof}

Theorem \ref{distancetrees} follows from Lemma \ref{distanceskeylemma} and Theorem \ref{mechanism}.  Note that when we apply Theorem \ref{mechanism}, we can start with any skeleton $x^1,\dots,x^{k+1}\in K_1\times K_2$ provided no two points $x^i,x^j$ share a coordinate.

\subsection{\textit{Proof of Theorem \ref{phi_distancetrees}}}\label{general_proof}
As in the previous two sections, Theorem \ref{phi_distancetrees} on $\phi$ distance trees
is an immediate consequence of Theorem \ref{mechanism} and the following lemma.  
\begin{lem}[Pin Wiggling for $\phi$ distance trees]
\label{phi_distanceskeylemma}
Let $K_1,K_2$ be Cantor sets satisfying $\tau(K_1)\cdot \tau(K_2) > 1$
Suppose $\phi:\R^2 \times \R^2 \rightarrow \R$ 
satisfies the derivative condition of Definition \ref{deriv_cond} on $A\times B$ for open $A,B\subset \R^2$, each of which intersects $K_1\times K_2$. 
For any $x^0 \in A$, 
there exists an open $S$ about $x^0$ so that
$$\bigcap_{x\in S} \Delta_{\phi, x}(K_1\times K_2)$$
has non-empty interior, 
where $\Delta_{\phi, x}(K_1\times K_2) = \{ \phi(x,y): y \in K_1\times K_2\}.$
\end{lem}
The strategy for establishing Lemma \ref{phi_distanceskeylemma} is as follows:
Given a pin $x\in \R^2$ and distance $t\in \R$, we note that $t\in \Delta_{\phi,x}(K_1\times K_2)$ whenever 
$\phi(x,y) = t$ for some $y=(y_1,y_2)\in K_1\times K_2$.  
We then use the implicit function theorem to solve for $y_2$ in terms of $(x, y_1,t)$ and call the resulting function $g$.  
Observing that $g$ behaves like the function $g_{x,t}$ introduced in the proof of Lemma \ref{distanceskeylemma} from the previous section, 
the lemma then follows from the exact proof used for Euclidean distances.  The only real effort of the proof then is setting up the implicit function theorem.  
\\

\begin{proof}
Let $\varphi_x(y) = \phi(x,y)$ as in Definition \ref{deriv_cond} so that for all $(x,y) \in A\times B$, 
\begin{equation}\label{deriv_varphi}
\frac{\partial (\varphi_x)}{\partial y_i} (y)\neq 0  \text{ for } i =1,2.
\end{equation}
Define $F:\R^2\times  \R^2 \times \R \rightarrow \R$ by 
\begin{equation}\label{Feq} 
F(x, y,t) = \phi(x,y) -t.
\end{equation}
Then $F$ is continuously differentiable on $A\times B\times \R$.  Moreover, it follows by \eqref{deriv_varphi} that the partial derivative of $F$ in $y_2$  is non-vanishing on $A\times B\times \R$.  
Choose 
$x^0 \in A$ and $u=(u_1,u_2) \in B\cap  (K_1 \times K_2)$. Set
$$t_0 = \phi( x^0, u)$$
so that
 $$F(x^0,u, t_0) = 0.$$

By the implicit function theorem, 
there exists a real-valued function $g$ and a $\de$-ball about the point $(x^0, u_1, t_0)$, denoted by 
$B_\de= B_{\de}(x^0, u_1, t_0)$, so that 
\begin{itemize}
\item $g$ is continuously differentiable on $B_{\de}$;
\item $g(x^0,u_1, t_0) = u_2$;
\item if $(x, y_1, t) \in B_{\de}$, then $(x, y) \in A\times B$ for all $y= (y_1, g(x,y_1, t) )$; 
\item $\phi(x, y) = t$ for all $y= (y_1, g(x,y_1, t) ) $ when $(x, y_1, t) \in B_{\de}$.
\end{itemize}
Moreover, if $g_{x,t}(y_1) = g(x,y_1, t)$, then 
$$g_{x,t}' (y_1) = 
 	 -\left( \frac{\partial (\varphi_x)}{\partial y_1}(y_1, g(x, y_1, t))\right) / 
 	 \left( \frac{\partial (\varphi_x)}{\partial y_2}(y_1, g(x, y_1, t))\right) \neq 0 
$$
for $(x, y_1, t) \in B_\de= B_{\de}(x^0,u_1, t_0)$.  
In other words, there is a neighborhood of $(x^0, t_0) $ so that the derivative of $g_{x,t}$ is non-vanishing in a neighborhood of $u_1$.

Replacing the function in \eqref{baby_g} of Lemma \ref{distanceskeylemma} with this new $g_{x,t}$, the proof proceeds as in the proof of Lemma \ref{distanceskeylemma}.  
\end{proof}


\subsection{\textit{Distance Trees in the Middle Thirds Cantor Set}}
In this section, we prove Theorem \ref{trees_thrm_middle_third}.  
The new twist in this context is that the middle thirds Cantor set has thickness equal to 1, not strictly greater than 1.  
Our method uses Lemma \ref{imagethickness} which controls how much smaller the thickness of $g(\widetilde{K})$ can be compared to $K$, where $\widetilde{K}$ is a choosen subset of $K$.  
If $\tau(K)=1$, then we cannot assume $\tau(g(\widetilde{K}))\geq 1$ and therefore cannot apply the Newhouse Gap Lemma directly.  However, the self similarity of the middle thirds Cantor set provides a tool that allows one to adapt the proof of the Newhouse Gap Lemma in this special setting.  This observation was used in \cite{STinterior} to prove that the pinned distance set $\Delta_x(C_{1/3}\times C_{1/3})$ has non-empty interior.  We use this idea to prove a pin wiggling lemma for the middle thirds Cantor set. \\

 Before proceeding, we introduce some terminology.  Let $C_{1/3}$ denote the standard middle thirds Cantor set in the interval $[0,1]$.  The standard construction of this set is given by defining $C_{1/3}$ as the intersection of a family of sets $\{C_n\}$, where each $C_n$ is the union of $2^n$ closed intervals of length $1/3^n$.  Define a \textbf{section} of $C_{1/3}$ to be the intersection of $C_{1/3}$ with any one of the intervals making up any of the sets $C_n$. 
\begin{lem}
\label{CantorSetImage}
Let $K_1,K_2$ be sections of the standard middle thirds Cantor $\mathcal{C}$, and let $g$ be a continuously differentiable monotone function which satisfies $1<g'<3$ on the convex hull of $K_1$.  If $K_2$ and $g(K_1)$ are linked, then $K_2\cap g(K_1)\neq \emp$.
\end{lem}

\begin{proof}
Let $I_1,I_2$ denote the convex hulls of $K_1, K_2$, respectively.
  By the mean value theorem, for any interval $J$ there exists $x_J\in J$ such that $|g(J)|=|J|\cdot|g'(x_J)|$.  In particular, if $U$ and $V$ are bounded gaps of $K_1$ and $|U|>|V|$, we have $|U|\geq 3|V|$ and therefore
\begin{align*}
|g(U)|&=|U|\cdot |g'(x_U)| \\
&> 3|V|\cdot 1 \\
&> |g'(x_V)|\cdot |V| \\
&= |g(V)|.
\end{align*}
It follows that the bridges of the images are the images of the bridges.  More precisely, the bridge of gap $g(U)$ in $g(K_1)$ is $g(B)$, where $B$ is the bridge next to the gap $U$. \\

To prove the theorem, we proceed by contradiction and assume $K_2\cap g(K_1)=\emp$.  The strategy of the proof is a variant of the strategy used to prove the original Newhouse Gap Lemma.  We construct sequences of gaps $U_n, V_n$ satisfying the following conditions:
\begin{itemize}
	\item $U_n$ is a bounded gap of $K_2$ and $g(V_n)$ is a bounded gap of $g(K_1)$
	\item $U_n$ and $g(V_n)$ are linked for every $n$
	\item For every $n$, either
	\begin{enumerate}[(a)]
		\item $U_{n+1}=U_n$ and $|g(V_{n+1})|<|g(V_n)|$, or
		\item $V_{n+1}=V_n$ and $|U_{n+1}|<|U_n|$
	\end{enumerate}
\end{itemize}
Thus, at each stage we are replacing one of the gaps $U_n,g(V_n)$ with a strictly smaller one and leaving the other unchanged.  In particular, this means that one of the two gap sequences must have a subsequence which is strictly decreasing in length.  Since gaps of different sizes are automatically disjoint and the total length of the gaps is bounded, we must have either $|U_n|\to 0$ or $|g(V_n)|\to 0$ as $n\to\infty$.  However, because $U_n$ and $g(V_n)$ are linked, the closure of $U_n$ contains points of both $K_2$ and $g(K_1)$, and similarly for $g(V_n)$.  This implies that the distance between $K_2$ and $g(K_1)$ is zero, which contradicts the assumption $K_2\cap g(K_1)= \emp$. \\

We construct our sequences $\{U_n\}$ and $\{V_n\}$ recursively.  First, we construct $U_1$ and $V_1$.  Recall that $I_i$ denotes the convex hull of $K_i$ for $i=1,2$.  Since $I_2$ and $g(I_1)$ are linked by assumption, there is a bounded gap $U_1$ of $K_2$ which is a subset of $I_2\cap g(I_1)$.  Let $u$ be an endpoint of $U_1$.  In particular we have $u\in K_2$, and therefore, by our assumption that $K_2\cap g(K_1)=\emp$, it follows that $u$ is in some bounded gap $g(V_1)$ of $g(K_1)$.  
 \\

\begin{figure}
\begin{tikzpicture}

\draw[ultra thick] (0,0)--(5,0);
\draw[ultra thick, blue] (5,0)--(6,0);
\draw[ultra thick, red] (6,0)--(8,0);
\draw[ultra thick] (8,0)--(12,0);

\draw [fill, blue] (5,0) circle [radius=0.1];
\draw [fill, blue] (6,0) circle [radius=0.1];
\draw [fill,] (8,0) circle [radius=0.1];

\node [below, blue] at (5.5,0) {$g(B_n^V)$};
\node [below, red] at (7,0) {$g(V_n)$};

\draw[ultra thick] (0,1)--(4,1);
\draw[ultra thick, red] (4,1)--(7,1);
\draw[ultra thick, blue] (7,1)--(10,1);
\draw[ultra thick] (10,1)--(12,1);

\draw [fill, blue] (7,1) circle [radius=0.1];
\draw [fill, blue] (10,1) circle [radius=0.1];
\draw [fill] (4,1) circle [radius=0.1];

\node [above, red] at (5.5,1) {$U_n$};
\node [above, blue] at (8.5,1) {$B_n^U$};

\end{tikzpicture}
\caption{Construction of gap sequence (gaps are red, bridges are blue)}
\label{gapsequence}
\end{figure}
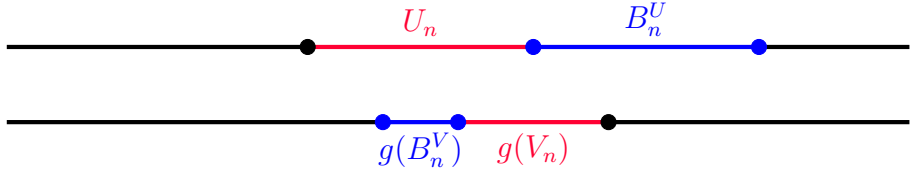

Next, suppose we have constructed $U_n$ and $V_n$ satisfying the properties in the first two bullet points above.  We construct $U_{n+1}$ and $V_{n+1}$ satisfying the third bullet point.
By construction, $U_n$ and $g(V_n)$ are linked.  Let $B_n^U$ be the bridge corresponding to $U_n$ on the same side as $g(V_n)$ (see Figure \ref{gapsequence}), and let $g(B_n^V)$ be the bridge of $g(V_n)$ on the same side as $U_n$.  We claim that one of the following two inequalities must hold:
\begin{align*}
|g(B_n^V)|&>|U_n| \\
|B_n^U|&>|g(V_n)|.
\end{align*}
This follows because we are working with the middle thirds Cantor set, so bounded gaps and corresponding bridges have the same length, i.e., we have $|B_n^U|=|U_n|$ and $|B_n^V|=|V_n|.$
The middle thirds Cantor set also has the property that if $U$ and $V$ are gaps with $|U|>|V|$, we automatically have $|U|\geq 3|V|$.  Finally, by our assumptions on $g$, we have $|J|< |g(J)|< 3|J|$ for every interval $J\subset I_1$.  The claim then follows, as $|U_n|>|V_n|$ implies the second inequality and $|U_n|\leq |V_n|$ implies the first.  \\

Assume the second inequality holds (an analogous argument will apply when the first holds).  This means one endpoint of $g(V_n)$ is in $U_n$ and the other is in $B_n^U$ (see Figure \ref{gapsequence} again).  Recall that we are proving the lemma by contradiction, assuming $K_2\cap g(K_1)=\emp$.  This assumption means the endpoint of $g(V_n)$ which is in $B_n^U$ must be contained in some gap $U_{n+1}$ of $K_2$, and by definition of a bridge we must have $|U_{n+1}|<|U_n|$.  We then take $V_{n+1}=V_n$.  This completes the construction.
\end{proof}

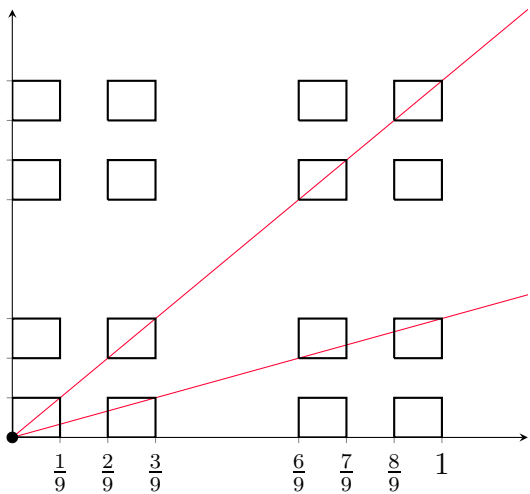
\begin{figure}[ht]
\centering
\begin{tikzpicture}
\begin{axis}[xmin=0,xmax=1.2,ymin=0,ymax=1.2,axis x line=center, axis y line=center, xtick={1/9,2/9,3/9,6/9,7/9,8/9,1}, xticklabels={$\frac{1}{9}$,$\frac{2}{9}$,$\frac{3}{9}$,$\frac{6}{9}$,$\frac{7}{9}$,$\frac{8}{9}$,$1$}, ytick={1/9,2/9,3/9,6/9,7/9,8/9,1},yticklabels={}]

\addplot[mark=*] coordinates {(0,0)};

\addplot[red] coordinates {(0,0)(3,3)};
\addplot[red] coordinates {(0,0)(3,1)};

\addplot[thick] coordinates {(0,0)(0,1/9)(1/9,1/9)(1/9,0)(0,0)};
\addplot[thick] coordinates {(2/9,0)(2/9,1/9)(3/9,1/9)(3/9,0)(2/9,0)};
\addplot[thick] coordinates {(0,2/9)(0,3/9)(1/9,3/9)(1/9,2/9)(0,2/9)};
\addplot[thick] coordinates {(2/9,2/9)(2/9,3/9)(3/9,3/9)(3/9,2/9)(2/9,2/9)};

\addplot[thick] coordinates {(6/9,0)(6/9,1/9)(7/9,1/9)(7/9,0)(6/9,0)};
\addplot[thick] coordinates {(8/9,0)(8/9,1/9)(9/9,1/9)(9/9,0)(8/9,0)};
\addplot[thick] coordinates {(6/9,2/9)(6/9,3/9)(7/9,3/9)(7/9,2/9)(6/9,2/9)};
\addplot[thick] coordinates {(8/9,2/9)(8/9,3/9)(9/9,3/9)(9/9,2/9)(8/9,2/9)};

\addplot[thick] coordinates {(0,6/9)(0,7/9)(1/9,7/9)(1/9,6/9)(0,6/9)};
\addplot[thick] coordinates {(2/9,6/9)(2/9,7/9)(3/9,7/9)(3/9,6/9)(2/9,6/9)};
\addplot[thick] coordinates {(0,8/9)(0,9/9)(1/9,9/9)(1/9,8/9)(0,8/9)};
\addplot[thick] coordinates {(2/9,8/9)(2/9,9/9)(3/9,9/9)(3/9,8/9)(2/9,8/9)};

\addplot[thick] coordinates {(6/9,6/9)(6/9,7/9)(7/9,7/9)(7/9,6/9)(6/9,6/9)};
\addplot[thick] coordinates {(8/9,6/9)(8/9,7/9)(9/9,7/9)(9/9,6/9)(8/9,6/9)};
\addplot[thick] coordinates {(6/9,8/9)(6/9,9/9)(7/9,9/9)(7/9,8/9)(6/9,8/9)};
\addplot[thick] coordinates {(8/9,8/9)(8/9,9/9)(9/9,9/9)(9/9,8/9)(8/9,8/9)};

\end{axis}

\end{tikzpicture}
\caption{The wedge of acceptable boxes}
\label{wedge}
\end{figure}

\begin{thm}
Given a point $x^0\in \R^2$, let $W_{x^0}$ denote the open wedge
\[
W_{x^0}=\{x\in\R^2:(x_2-x_2^0)<(x_1-x_1^0)<3(x_2-x_2^0)\}.
\]
Let $K_1,K_2$ be sections of the middle thirds Cantor set, and suppose $K_1\times K_2\subset W_{x^0}$ (Figure \ref{wedge}).  Then, there exists an open neighborhood $S$ of $x^0$ such that
\[
\bigcap_{x\in S}\Delta_x(K_1\times K_2)
\]
has non-empty interior.
\end{thm}

\begin{proof}
For $x\in\R^2$ and $t\in\R$, consider the function
\[
g_{x,t}(z)=x_2+\sqrt{t^2-(z-x_1)}.
\]
We have
\[
g_{x,t}'(z)=-\frac{z-x_1}{g_{x,t}(z)-x_2},
\]
so $1<|g_{x,t}'(z)|<3$ whenever $(z,g_{x,t}(z))\in W_x$.  For $j=1,2$, let $u_j=\min K_j$, so that $u=(u_1,u_2)$ denotes the lower left corner of $K_1\times K_2$, and let $t_0=|x^0-u|$ (thus, we have $g_{x^0,t_0}(u_1)=u_2$).  For every $\delta>0$, define $\widetilde{K_j}=K_j\cap [u_1,u_1+\delta]$ and consider the set
\[
U_\delta=\{(x,t) :\widetilde{K_2}\text{ and }g_{x,t}(\widetilde{K_1}) \text{ are linked, and }(z,g_{x,t}(z))\in W_x \text{ for all }z\in[u_1,u_1+\delta]\}.
\]
By Lemma \ref{CantorSetImage}, if $(x,t)\in U_\delta$ then $t\in \Delta_x(K_1\times K_2)$.  Clearly $U_\delta$ is open, so as in the previous proofs it suffices to show that $U_\delta$ contains a point of the form $(x^0,t)$.  We first observe that if $\delta$ is sufficiently small, the sets $\widetilde{K_2}$ and $g_{x,t}(\widetilde{K_1})$ are linked for all $t\in(t_0,t_0+\delta)$.  This is because for any $t>t_0$ we have $g_{x^0,t}(u_1)>u_2$, and for any $t<u_1+\delta$ we will also have $g_{x^0,t}(u_1)<u_2+\delta$ and $g_{x^0,t}(u_1+\delta)<u_2$.  Finally, since $u\in W_{x^0}$ by assumption, if $\delta'$ is sufficiently small we will have $(z,g_{x^0,t}(z))\in W_{x^0}$ for all $z\in [u_1,u_1+\delta]$ and all $t\in (t_0,t_0+\delta')$.

\end{proof}

\begin{proof}[Proof of Theorem \ref{trees_thrm_middle_third}]
By Theorem \ref{mechanism}, it suffices to find a skeleton $x^1,...,x^{k+1}\in C_{1/3}\times C_{1/3}$ such that whenever $i<j$ we have $x^j\in W_{x^i}$.  Wedge membership is transitive in the sense that $x^3\in W_{x^2}$ and $x^2\in W_{x^1}$ implies $x^3\in W_{x^1}$, so it suffices to construct our skeleton so that $x^{i+1}\in W_{x^i}$ for every $i$.  We construct such a sequence recursively as follows.  Let $x^1$ be the origin.  One can check that $W_{x^1}$ contains the box $[8/9,1]\times [6/9,7/9]$ (refer again to Figure \ref{wedge}).  Let $x^2=(8/9,6/9)$.  Since $x^2$ is the lower left corner of a similar copy of $C_{1/3}\times C_{1/3}$, one can run the same argument by symmetry and take $x^3$ to be the lower left corner of a box contained in $W_{x^2}$.  This process can be repeated as many times as needed.
\end{proof}

\bibliography{refs}
\bibliographystyle{abbrv}

\end{document}